\documentclass[11pt]{article}

\usepackage{amsthm}
\usepackage{amsmath}
\usepackage{amsfonts}
\usepackage{amssymb}
\usepackage{amscd}
\usepackage{mathrsfs}
\usepackage[all]{xypic}
\date{}

\newtheorem{thm}{Theorem}[section]

\newtheorem{lem}[thm]{Lemma}
\newtheorem{prop}[thm]{Proposition}

\newtheorem{alg}[thm]{Algorithm}

\theoremstyle{definition}

\theoremstyle{definition}
\newtheorem{example}[thm]{Example}

\theoremstyle{remark}

\newcommand{\A}{\mathbf{A}}

\newcommand{\C}{\mathbb{C}}

\newcommand{\Z}{\mathbb{Z}}
\newcommand{\R}{\mathbb{R}}
\newcommand{\Q}{\mathbb{Q}}

\newcommand{\idealB}{\mathfrak{B}}
\newcommand{\idealI}{\mathfrak{T}}  
\newcommand{\idealM}{\mathfrak{M}}
\newcommand{\idealJ}{\mathfrak{J}}
\newcommand{\idealP}{\mathfrak{P}}

\begin{document}
\bibliographystyle{plain}
\title{Siphons in chemical reaction networks}
\author{Anne Shiu and Bernd Sturmfels}

\maketitle

\begin{abstract}
Siphons in a chemical reaction system are subsets of the species
that have the potential of being absent in a  steady state.
We present a characterization of minimal siphons
in terms of primary decomposition of binomial ideals,
we explore the underlying geometry, and we demonstrate
the effective computation of siphons using computer algebra software.  
This leads to a new method for determining whether 
given initial concentrations allow for various boundary steady states.

  \vskip 0.3cm

  \noindent \textbf{Keywords:} chemical
  reaction systems, siphon, steady state, monomial ideal,
  binomial ideal, primary decomposition

\end{abstract}

\section{Introduction}
In systems biology, a population model or chemical reaction system is said to be ``persistent'' if none of its species can become extinct if all species are present at the initial time.    Those subsets of the species that can be absent in steady state are called ``siphons.''  
Angeli {\em et al.}~\cite{PetriNetExtended} suggested the concept of siphons to 
 study the long-term behavior of dynamical systems that model
 chemical reactions.
 In terms of the dynamics,
                a siphon is the index set of a forward-invariant face of the positive orthant.  Any boundary steady state must lie in the interior of such a face. Hence, to investigate the trajectories, it 
        is useful to list all minimal siphons.
        The present paper offers an algebraic characterization of siphons,  and it
        shows how this  translates into a practical tool for computing siphons.
                
Following \cite{Adleman, TDS}, we represent a chemical reaction network 
as a directed graph $G$ whose nodes are labeled by monomials and whose edges correspond to reactions. A {\em siphon}
of $G$ is a non-empty subset $Z$ of the variables such that, for every
directed edge $m \rightarrow m'$ in $G$, whenever one of the variables
in the monomial $m'$ lies in $Z$ then so does at least one of the variables in~$m$.
In Section 2 we relate this definition to the description of siphons given in
\cite{PetriNetExtended, EnumSiphons}, we review the underlying dynamics, and
we discuss its meaning in terms of polyhedral geometry.
Our algebraic approach is presented in Section 3.  
Theorem~\ref{thm:Main} expresses the minimal
siphons of $G$ in terms of the primary decomposition of a binomial ideal
associated to $G$. If the directed graph $G$ is strongly connected then the ideal
encoding the minimal siphons is generated by squarefree monomials.
In Theorem~\ref{thm:SecondMain} and Algorithm~\ref{algorithm}
we explain how to compute the relevant (stoichiometrically compatible) siphons for any set of initial 
conditions. In particular, a chemical reaction system without relevant siphons has no boundary steady states, and this property is sufficient for proving persistence in many systems \cite{PetriNetExtended, SiegelMacLean}.  
  In Section 4, we demonstrate that  the relevant computations can be performed 
  effectively using computer algebra software, such as
  {\tt Macaulay 2} \cite{M2}.

In the remainder of the Introduction we present three examples
from the systems biology literature, with the aim of illustrating
our algebraic representation of chemical reaction networks
and the computation of siphons.

\begin{example} \label{ex:1} 
We consider a receptor-ligand dimer model, which is analyzed by Chavez in her thesis \cite[\S 7.2]{ChavezThesis} and by Anderson \cite[Example 4.1]{Anderson08}:
\[ 
 \begin{xy}<15mm,0cm>:
 (0,1) ="c1^3" *+!DR{A^{2}C } *{}; 
 (1,1) ="c1c2^2" *+!DL{AD} *{}; 
 (1,0) ="c2^3" *+!UL{{E}} *{}; 
 (0,0) ="c1^2c2" *+!UR{BC} *{}; 
   {\ar "c1^2c2"+(-0.15,0)*{};"c1^3"+(-0.15,0)*{}};
   {\ar "c2^3"+(0.05,0)*{};"c1c2^2"+(0.05,0)*{}};
   {\ar "c1^3"+(-0.05,0)*{};"c1^2c2"+(-0.05,0)*{}};
   {\ar "c1c2^2"+(+0.15,0)*{};"c2^3"+(+0.15,0)*{}};
   {\ar "c1^3"+(0,+0.15)*{};"c1c2^2"+(0,+0.15)*{}};
   {\ar "c1^2c2"+(0,-0.05)*{};"c2^3"+(0,-0.05)*{}};
   {\ar "c1c2^2"+(0,0.05)*{};"c1^3"+(0,0.05)*{}};
   {\ar "c2^3"+(0,-0.15)*{};"c1^2c2"+(0,-0.15)*{}};
 (0.5,1.18)="k12" *+!D{\kappa_{12}} *{};
 (0.5, 0.8)="k21" *+!D{\kappa_{21}} *{};
 (1.21,.5)="k23" *+!L{\kappa_{23}}*{};
 (0.6,.5)="k32" *+!L{\kappa_{32}}*{};
 (.5,-.3)="k34" *+U{\kappa_{34}}*{};
 (.5,.1)="k43" *+U{\kappa_{43}}*{};
 (-.33,.5)="k41" *+R{\kappa_{41}}*{};
 (.2,.5)="k14" *+R{\kappa_{14}}*{};
    \end{xy}
 \]
 Note that the reaction $A^{2}C \leftrightarrows {A} {D} $ is usually denoted by $2A+C \leftrightarrows A + D$.  
The biochemical species are as follows: the species $A$ denotes a receptor, $B$ denotes a ``dimer'' state of $A$ (two receptors joined together), and $C$ denotes a ligand that can bind either to $A$ (to form $D$) or to $B$ (to form $E$).  There are three minimal siphons, $\{A, B, E\}$, $\{A,C,E\}$, and $\{C,D,E\}$, which correspond to the minimal primes of the monomial ideal of the complexes $\langle {A}^{2}{C}, ~  {A}{D},  ~ {E}, ~ {B}{C} \rangle$.  We will return to this example in Section~\ref{sec:ComputingSiphons}. \qed
\end{example}

\begin{example}\label{ex:2}
The following enzymatic mechanism was analyzed by Siegel and
MacLean \cite{SiegelMacLean}, and also by Chavez 
 \cite[Example 4.6.1]{ChavezThesis}: 
\begin{align*}
S E ~ \leftrightarrows~   Q~ &  \leftrightarrows~  P E \\
Q   I ~ & \leftrightarrows~ R~.
\end{align*}
The species are $S$ (a substrate), $E$ (an enzyme), $P$ (a product),
$I$ (an uncompetitive inhibitor), and intermediate
complexes $Q$ and $R$.
 Here the graph consists of two strong components,
and we encode it in the binomial ideal
$\langle S  E-Q, ~ Q - P  E,  ~ Q  I-R \rangle + \langle EPQRS \rangle$.
The radical of this ideal equals
$$ \langle E,~Q,~ R \rangle \,\,\cap \,\,
 \langle I,~R,~E  S-Q,~P-S \rangle \,\, \cap \,\,
 \langle P,~Q,~R,~S \rangle ~.
 $$
 By Theorem~\ref{thm:Main}, the  minimal siphons are the variables in
 these prime ideals. Thus the minimal siphons are
   $\{ E,~Q,~R \}$, $\{ I,~ R \}$, and $\{ P,~Q,~R,~S \}$.   \qed
\end{example}

\begin{example}\label{ex:3}
Here is the network for a basic one-step conversion reaction:
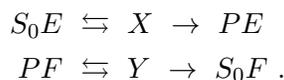
\begin{align*}
S_0 E ~ & \leftrightarrows~  X ~ \rightarrow~  P E \\
P F ~& \leftrightarrows~ Y ~  \rightarrow~S_0 F~.
\end{align*}
The enzyme $E$ helps convert a
substrate $S_0$ into a product $P$, and a second enzyme $F$ reverts
the product $P$ back into the original enzyme $S_0$; these are also called ``futile cycles'' \cite{AngeliSontag06,GeoMulti}.  
Such reactions include phosphorylation and de-phosphorylation events,
and they take place in MAPK cascades. This network has 
three minimal siphons: $\{ E,~X\}$, $\{F,~Y \}$, and
$\{ P, ~S_{0}, ~X, ~Y \}$.
To see this algebraically, we form the binomial ideal
\begin{align*}
\bigl\langle E S_0-X,~X  (E P-X),~ F  P - Y, ~Y  (F  S_0-Y),~E  F  P  S_0  X  Y 
\bigr\rangle~.
\end{align*}
This ideal corresponds to  $\idealI_G$ in Theorem~\ref{thm:Main}, and it has six minimal primes:
\begin{align*}
\langle E,~X, ~ F  S_0  - Y,~  P-S_0 \rangle, \,\,
\langle F,~ Y,~ P-S_0,~  S_0 E  - X \rangle, \,\,
\langle P,~S_0,~X, ~Y \rangle , \\
\langle E,~ X, ~ Y,~ F \rangle , \quad
\langle E,~X,~P,~ Y \rangle, \quad  \text{and} \quad
\langle F,~ S_0, ~X,~ Y \rangle ~ .
\end{align*}
The three minimal siphons arise from the first three of these six primes.
\qed
\end{example}

\section{Reaction networks, siphons, and steady states}  \label{sec:siphon}

A {\em chemical reaction network} is defined by a finite labeled directed graph $G$
with $n$ vertices. The $i$-th vertex of $G$ is labeled with a monomial
$ c^{y_i} \,\,\, = \,\,\, c_1^{y_{i1}} c_2^{y_{i2}} \cdots  c_s^{y_{is}}$ in 
$s$ unknowns $c_1,\ldots,c_s$, 
and an edge $(i,j)$ is labeled by a positive parameter $\kappa_{ij}$.  This graph defines the ordinary differential equations
\begin{equation}
\label{CRN}
 \frac{d c}{dt} \quad = \quad \Psi(c) \cdot A_\kappa \cdot Y~,
 \end{equation} 
where 
$ \Psi(c) ~ = ~ \bigl( c^{y_1}, \, c^{y_2} , \, \ldots, \,c^{y_n} \bigr)$ is the row vector of
the monomials, 
$Y=(y_{ij})$ is the $n \times s$-matrix of exponent vectors of the $n$ monomials,
and $A_\kappa$ is the $n \times n$-matrix 
whose off-diagonal entries
are the $\kappa_{ij}$ and whose row sums are zero
(i.e.~minus the Laplacian of $G$).  The equations (\ref{CRN}) are those of {\em mass-action kinetics}, although the concept of a siphon is independent of the choice of kinetics.  
In order for each chemical {\em complex}  $c^{y_i}$ to be a reactant or product of at least one reaction, we assume that $G$ has no isolated points.  For a complex $c^{y_i}$ and for $a \in [s]$, we write $c_a | c^{y_i}$ (``$c_a$ divides $c^{y_i}$'') if $y_{ia}>0$; in other words the $i$-th complex contains species $a$.  If the $i$-th complex does not contain species $a$, then we write $ c_a \nmid c^{y_i} $.

A non-empty subset $Z$ of the index set $[s] := \{1,2,\ldots,s\}$
is a {\em siphon} if for all $z \in Z$ and all reactions
 $c^{y_{i}} \longrightarrow c^{y_{j}}$ with $c_{z} |  c^{y_{j}}\,$,
 there exists $a \in Z$ such that $c_{a} | c^{y_{i}}$.
Siphons were called ``semilocking sets'' in \cite{Anderson08,AndShiu09}.  Note that the set of siphons 
of $G$ does not depend on the choice of parameters $\kappa_{ij}$.
 
With any non-empty subset $Z \subset [s]$ we associate the prime ideal 
\begin{align*}
\idealP_{Z} ~ := ~ \langle \, c_{a} ~:~ a \in Z \,\rangle
\end{align*}
in the polynomial ring $\mathbb{Q}[c_{1}, c_{2}, \dots, c_{s}]$.  Recall 
(e.g.~from \cite{CLO}) that the {\em variety} of $\idealP_Z$, denoted by $V(\idealP_{Z})$, is the set of points $x \in \mathbb{R}^s$ such that $f(x)=0$ for all polynomials $f \in \idealP_Z$.  
Thus, the non-negative variety $V_{\geq 0} (\idealP_{Z})$ 
is the face of the positive orthant $\R^s_{\geq 0}$ defined by all $Z$-coordinates being zero.

\begin{prop} \label{AlgSiphon}
A non-empty subset $Z$ of $[s]$ is a siphon if and only if $V_{\geq 0} (\idealP_{Z})$
  is forward-invariant with respect to the dynamical system (\ref{CRN}).
\end{prop}
\begin{proof} This is the content of Proposition 2 in Angeli {\em et al.} \cite{PetriNetExtended}. 
\end{proof} 

In Example \ref{ex:1},  the dynamical system (\ref{CRN}) takes the explicit form
 \begin{align*}
& {d A/dt} ~=~ -(\kappa_{12} + 2 \kappa_{14}) A^{2} C + (\kappa_{21} -\kappa_{23}) A D + \kappa_{32} E +2 \kappa_{41} B C \\
& {d B/dt} ~=~ \kappa_{14} A^{2} C  -(\kappa_{41}+\kappa_{43}) B C +\kappa_{34} E \\
& {d C/dt} ~=~ -\kappa_{12} A^{2} C + \kappa_{21} A D - \kappa_{43} B C +\kappa_{34} E \\
& {d D/dt} ~=~ \kappa_{12} A^{2} C -(\kappa_{21} + \kappa_{23}) A D +\kappa_{32} E \\
& {d E/dt} ~=~ \kappa_{23} A D + \kappa_{43} B C - (\kappa_{32}+ \kappa_{34})  E~.
\end{align*}
This is a dynamical system on  $\R^5_{\geq 0}$.  Each of  the three minimal siphons $\{A,B,E\}$, $\{A,C,E\}$, and $\{C,D,E\}$ defines a two-dimensional face of $\mathbb{R}^{5}_{\geq 0}$.  For example,  $V_{\geq 0}(\idealP_{\{ A,B,E \}})$ is the face in which the coordinates $A$, $B$, and $E$ are zero and $C$ and $D$ are non-negative.  The minimality of the three siphons implies that no face of dimension three or four is forward-invariant.

We next collect some results relating siphons to 
{\em boundary steady states}, that is, non-negative
steady states of (\ref{CRN}) having at least one zero-coordinate.  
These connections are behind our interest  in computing siphons. 
See \cite{Anderson08,AndShiu09,PetriNetExtended}
 for details on how siphons relate to questions of {\em persistence} (the property that positive trajectories of (\ref{CRN}) have no accumulation points on the boundary of the orthant  $\R_{\geq 0}^s$).
 We first show that a boundary steady state necessarily lies in the relative interior of a face 
  $V_{\geq 0} (\idealP_{Z})$ indexed by a siphon $Z$.

\begin{lem}\label{lem:bdySiphon}
Fix a reaction network $G$, and let $\gamma$ be a point on the boundary of the positive orthant $\mathbb{R}_{\geq 0}^{s}$ with zero coordinate set $\,Z:= \{\, i \in [s] : \gamma_{i}=0 \, \}$.
 If $\gamma$ is a boundary steady state of (\ref{CRN}), then the index set $Z$ is a siphon.
\end{lem}

\begin{proof}
Assume that $c_{z} | c^{y_{j}}$ for some species $z\in Z$ and some complex $c^{y_{j}}$ of $G$.  Let $\mathcal{I}$ index complexes that react to $c^{y_{j}}$ but do not contain the species~$z$:
$$ \mathcal{I} \,\, := \,\,
\bigl\{ i\in [n] ~:~ c^{y_{i}} \longrightarrow c^{y_{j}} \text{ is a reaction of } G \text{ and } c_{z}\not| 
\, c^{y_{i}} \bigr\}. $$
  Then we have 
\begin{align} \label{eqn:ZeroDeriv}
 \frac{dc_{z}}{dt} \Bigr\rvert_{ c = \gamma } \quad = \quad \sum\limits_{i \in \mathcal{I}} \kappa_{ij} y_{jz} \gamma^{y_{i}}  \quad = \quad 0~, 
 \end{align}
where the second equality holds because $\gamma$ is a steady  state.  
The summands of (\ref{eqn:ZeroDeriv}) are non-negative, so we have $\gamma^{y_{i}}=0$ for all $i \in \mathcal{I}$.
Thus if $i \in \mathcal{I}$ there exists
$a_{i}\in [s]$ with $\gamma_{a_{i}}=0$ (so, $a_{i}\in Z$) and hence $c_{a_{i}}| c^{y_{i}}$.
\end{proof}

A similar result holds for boundary $\omega$-limit points (accumulation points) of a trajectory; see \cite{Anderson08, PetriNetExtended} or \cite[Theorem 2.13]{AndShiu09}.  We are interested in the
dynamics arising from some initial condition $c^{(0)} \in \mathbb{R}^{s}_{ > 0}$, so we restrict our attention to polyhedra, called {\em invariant polyhedra}, of the following form:
\begin{equation}
\label{invpoly}
 \mathcal{P}_{c^{(0)}}~:= ~ \bigl(\,c^{(0)} + L_{\rm stoi}\, \bigr)  \,\cap \, \mathbb{R}^{s}_{\geq 0}~.
 \end{equation}
Here $L_{\rm stoi}~:=~ \text{span} \{ y_{j}- y_{i} ~:~ c^{y_{i}} \rightarrow c^{y_{j}} \text{ is a reaction} \} $ is the {\em stoichiometric subspace} in $\mathbb{R}^{s}$. 
For any $c$ and any $\kappa$,
 the right hand side vector $\Psi(c) \cdot A_\kappa \cdot Y$ of
 the dynamical system (\ref{CRN}) 
  lies in $L_{\rm stoi}$, and therefore the polyhedron  $\mathcal{P}_{c^{(0)}}$ is forward-invariant with respect to (\ref{CRN}).  For any 
index set $W\subset [s]$, let 
\begin{align*}
F_{W} \,\,:= \,\, \{ \, x\in \mathcal{P}_{c^{(0)}} ~:~ x_{i}=0 \text{ if } i \in W  \}  \,\,= \,\, V_{\geq 0}(\idealP_{W}) \cap \mathcal{P}_{c^{(0)}}
\end{align*}
denote the corresponding (possibly empty) face of $\mathcal{P}_{c^{(0)}}$.  All faces of 
$\mathcal{P}_{c^{(0)}}$ have this form; see \cite[\S 2.3]{AndShiu09} for further details.  
Lemma~\ref{lem:bdySiphon} implies the following: {\em 
Given an invariant polyhedron $\mathcal{P}_{c^{(0)}}$, if all siphons $Z$ yield empty faces, $F_{Z}= \emptyset$, then $\mathcal{P}_{c^{(0)}}$ contains no boundary steady states.}  
In Theorem \ref{thm:all_not_relevant} 
we shall present an algebraic method for deciding when this happens.

We now examine the case when the  chemical reaction network
is strongly connected, i.e., between any two complexes
there is a sequence of reactions.

\begin{lem} \label{bdy=siphonStrConn}
Assume that $G$ is strongly connected. Then a point
$\gamma \in \R_{\geq 0}^s$  is a boundary steady state if and only if 
$Z = \{i \in [s] ~:~ \gamma_{i}=0 \}$ is a siphon.
\end{lem}

\begin{proof}
The forward implication  is  Lemma~\ref{lem:bdySiphon}.  
Now let $\gamma$ be a boundary point whose zero-coordinate set $Z$ is a siphon. 
Because $G$ is strongly connected,
 all complexes $c^{y_{i}}$ evaluated at $\gamma$ are zero ($\gamma^{y_{i}}=0$), and hence, each monomial that appears on the right-hand side of (\ref{CRN}) vanishes
 at $c=\gamma$.  
\end{proof}

From a polyhedral geometry point of view, Lemma~\ref{bdy=siphonStrConn} states the following: {\em For strongly connected reaction networks $G$, any face of an invariant polyhedron $\mathcal{P}_{c^{(0)}}$ either has no steady states in its interior or the entire face consists of steady states.}  
We shall see now that a similar result holds for toric dynamical systems.  
Recall (from~\cite{TDS})
 that (\ref{CRN}) is a {\em toric dynamical system} if the parameters $\kappa_{ij}$
are such that $\Psi(c) \cdot A_\kappa = 0$ has a positive solution $c \in \mathbb{R}^{s}_{>0}\,$
(which is called a {\em complex-balancing steady state}).
The following result concerns the faces of invariant polyhedra of toric dynamical systems.

\begin{lem}
  \label{lem:classification} 
  Let $c^{(0)} \in \R^{s}_{>0}$ be a positive initial condition of a toric dynamical system.  
  Then a face $F_{Z}$ of the
  invariant polyhedron $\mathcal{P}_{c^{(0)}}$ contains a steady state in its interior if and only if
  $Z$ is a siphon.  \end{lem}

\begin{proof}
This is derived in \cite[Lemma 2.8]{Anderson08}, albeit in different language.
\end{proof}

\section{Binomial ideals and monomial ideals}
\label{sec:Main}

In what follows we characterize the minimal siphons
of a chemical reaction network $G$ in the language of
combinatorial commutative algebra \cite{MS}. It will
be shown that they arise
as components in primary decompositions.  For any 
initial conditions $c^{(0)}$, we characterize those siphons that define
non-empty faces of the invariant polyhedron $\mathcal{P}_{c^{(0)}}$.
In the next section
we shall see that these results translate into a practical new method for enumerating siphons.

Throughout this section we fix the ring 
$\,R = \mathbb{Q}[c_{1}, \dots, c_{s}]/\langle c_1 c_2 \ldots c_s \rangle$.
This is the ring of polynomial functions with
$\Q$-coefficients on the union  of the
coordinate hyperplanes in $\R^s$.
All our ideals will live in this ring.   

With a given network $G$ we associate the following three ideals in $R$:
\begin{eqnarray*}
\idealI_G &  = & 
\bigl\langle \, c^{y_{i}} \cdot (c^{y_{j}} - c^{y_{i}})  ~:~ c^{y_{i}} \rightarrow c^{y_{j}} \text{ is a reaction of } 
G \,\bigr\rangle, \\
\idealJ_G &  = & 
\bigl\langle \, \,\, c^{y_{j}} - c^{y_{i}}  \, ~:~ \,\, c^{y_{i}} \rightarrow c^{y_{j}} \text{ is a reaction of } 
G \,\bigr\rangle, \\
\idealM_G & = & \bigl\langle \Psi(c) \bigr\rangle \,\,\, = \,\,\,
\bigl\langle \,c^{y_1} , c^{y_2}, \ldots, c^{y_n} \bigr\rangle .
\end{eqnarray*}
Thus $\idealI_G$ encodes the directed edges, and
$\idealJ_G$ encodes the underlying undirected graph.
These are pure difference binomial ideals \cite{DMM, ES},
while $\idealM_G$ is the monomial ideal of the complexes.
The following  is our main result.

\begin{thm}\label{thm:Main}
The minimal siphons of a chemical reaction network $G$ are the inclusion-minimal sets
$\{i \in [s]: c_i \in \idealP \}$ where $\idealP$ runs over the minimal primes of $\idealI_G$. 
If each connected component of $G$ is strongly connected then
$\idealI_G$ can be replaced in this formula by the ideal $\idealJ_G$.
Moreover, if $G$ is strongly connected then $\idealI_G$ can be replaced by
the monomial ideal $\idealM_G$. 
\end{thm}

\begin{proof}
The complex variety  $V_{\C}(\idealI_G)$ consists of all points $\gamma \in \C^s$
having at least one zero coordinate and satisfying
$\,\gamma^{y_{i}} \cdot (\gamma^{y_{j}} - \gamma^{y_{i}}) = 0$ for all reactions.
We first claim that 
our assertion  is equivalent to the statement that the minimal siphons are
the inclusion-minimal sets of the form $\{i \in [s]: \gamma_i = 0\}$
where $\gamma$ runs over $V_{\C}(\idealI_G)$.  
Indeed if $\idealP$ is a minimal associated prime 
of $\idealI_G$, let $\gamma \in \{0,1 \}^s$ be defined by $\gamma_i=1$ if and only if $c_i \notin \idealP$.   
It follows that $\{i \in [s]: c_i \in \idealP \} = \{i \in [s]: \gamma_i = 0\}$ and $\gamma \in V_{\C}(\idealP) \subset V_{\C}(\idealI_G)$.  Conversely, if $\gamma \in V_{\C}(\idealI_G)$, then $\gamma \in V_{\C}(\idealP)$ for some minimal associated prime $\idealP$, and so we have the containment $\{i \in [s]: c_i \in \idealP \} \subset \{i \in [s]: \gamma_i = 0\}$.  If, furthermore, the set $\{i \in [s]: \gamma_i = 0\}$ is minimal among those defined by $\gamma' \in V_{\C}(\idealI_G)$, then by above it must follow that the containment is in fact equality.

Next, if $\gamma$ is in $V_{\C}(\idealI_G)$
 then we can replace $\gamma$ by the $0$-$1$ vector $\delta$ defined by
$\delta_i = 0$ if $\gamma_i = 0$ and $\delta_i= 1$ if $\gamma_i \not= 0$.
This non-negative real vector has the same support as $\gamma$
and lies in the variety of $\idealI_G$. Hence our claim is that
 the minimal siphons are
the inclusion-minimal sets of the form $\{i \in [s]: \delta_i = 0\}$
where $\delta$ runs over  $V_{\{0,1\}}(\idealI_G)$. But this is obvious because
$\,\delta^{y_{i}} \cdot (\delta^{y_{j}} - \delta^{y_{i}}) = 0$
if and only if $\delta^{y_{j}} = 0$ implies $\delta^{y_{i}} = 0$.

Now, the minimal associated primes of $\idealI_G$ depend only on the radical of $\idealI_G$,
so we can replace $\idealI_G$ by any other ideal that has the same radical. If the 
components of 
 $G$ are strongly connected then the complex  $\,c^{y_{i}}\,$ can produce $ \,c^{y_{j}}\,$
 if and only if $\,c^{y_j}\,$ can produce $\,c^{y_i}$, and in this case
 both $\, c^{y_{i}} \cdot (c^{y_{j}} - c^{y_{i}})  \,$ and
 $\, c^{y_{j}} \cdot (c^{y_{i}} - c^{y_{j}})  \,$ are in $\idealI_G$. 
 Hence the radical
 of $\idealI_G$ contains the binomial  $\, c^{y_{i}} - c^{y_{j}}$, and we conclude that
 $\idealI_G$ and $\idealJ_G$ have the same radical.
 
 Finally, $\idealM_{G}$ is a monomial ideal, and associated primes of a monomial ideal are of the form $\idealP_{Z}$ for some $Z \subset [s]$.  It is straightforward to see that if $G$ is strongly connected, $\idealP_{Z}$ contains $\idealM_{G}$ if and only if $Z$ is a siphon.  
 \end{proof}

When analyzing a concrete chemical reaction network $G$, one often is given
an initial vector $c^{(0)} \in \R^s_{> 0}$ for the dynamical system (\ref{CRN}),
or at least a subset $\Omega$ of $\R^s_{>0}$ that contains $c^{(0)}$. A siphon $Z \subset [s]$ of $G$ is called {\em $c^{(0)}$-relevant}
if the face $F_Z$ of the invariant polyhedron $\mathcal{P}_{c^{(0)}}$ is non-empty.  In other words, if $Z$ is $c^{(0)}$-relevant, then there exists a boundary point that is stoichiometrically compatible with $c^{(0)}$ and has zero-coordinate set containing $Z$.  
For any subset $\Omega$ of $\R^s_{>0}$, we say that $Z$ is {\em $\Omega$-relevant}
if it is $c^{(0)}$-relevant for at least one point $c^{(0)}$ in $\Omega$. Finally we call 
a siphon {\em relevant} if it is $\R^s_{>0}$-relevant. 
Relevant siphons are also called ``critical'' siphons \cite{PetriNetExtended},  and non-relevant siphons are also called ``stoichiometrically infeasible'' siphons \cite{AndShiu09} and ``structurally non-emptiable'' siphons \cite{PetriNetExtended}.   
We next explain how to enlarge the ideals $\idealI_G$, $\idealJ_G$, and $\idealM_G$ so that
their minimal primes encode only the siphons that are relevant.

We recall that the {\em stoichiometric subspace} $L_{\rm stoi}$ of $\R^s$ is spanned by
all vectors $y_{j}- y_{i}$ where $c^{y_{i}} \rightarrow c^{y_{j}}$ is a reaction in $G$.
Its orthogonal complement $\, L_{\rm cons} := (L_{\rm stoi})^\perp \,$  is the space of {\em conservation relations}.  Let $\mathcal{Q}$ denote
the image of the non-negative orthant $\R^s_{\geq 0}$ in the quotient space
$\,\R^s/ L_{\rm stoi} \simeq L_{\rm cons}$. Thus $\mathcal{Q}$ is a convex polyhedral
cone and its interior points are in bijection with the invariant polyhedra $\mathcal{P}_{c^{(0)}}$. Further, $\mathcal{Q}$ is isomorphic to the cone spanned by the columns of any 
matrix $\A$ whose rows form a basis for $L_{\rm cons}$.  This isomorphism is given by the map
$$
\phi_{\A} ~  : ~ \mathcal{Q} ~ \rightarrow~ \biggl\{ \sum\limits_{i=1}^{s} \alpha_{i} a_{i} ~:~ \alpha_{1}, 
\alpha_{2}, \dots, \alpha_{s} \geq 0   \biggr\} \quad 
,
\quad \bar{q} ~\mapsto ~ \sum\limits_{i=1}^{s} q_{i} a_{i}   ~,
$$
where $q = (q_{1}, q_{2}, \dots, q_{s}) \in \mathbb{R}^{s}_{\geq 0}$ and $a_{1}, a_{2}, \dots, a_{s}$ 
are the columns of the matrix $\A$.  For simplicity, we identify the cone $\mathcal{Q}$ with the image of $\phi_{\A}$.
A subset $F$ of $[s] = \{1,2,\ldots,s\}$ is called a {\em facet} of $\mathcal{Q}$
if the corresponding columns of $\A$ are precisely the rays lying
on a maximal proper face of $\mathcal{Q}$.  Any maximal proper face of $\mathcal{Q}$ also is called a facet.  
The list of all facets of $\mathcal{Q}$ can be computed using
polyhedral software such as {\tt polymake} \cite{polymake}.

We represent the facets of $\mathcal{Q}$ by the following
squarefree monomial ideal:
$$ \idealB \quad = \quad \! \bigcap_{F \,{\rm facet} \,{\rm of}\,\mathcal{Q}} \!
\! \bigl\langle \,c_i \,: \, i \not\in F \, \bigr\rangle \quad = \quad \bigcap_{F \,{\rm facet} \,{\rm of}\,\mathcal{Q}} \!
\! \idealP_{F^{\rm c}} . $$
Each vertex of an invariant polyhedron $\mathcal{P}_{c^{(0)}}$
is encoded uniquely by its support $V $, which is a subset of $[s]$.
Consider the squarefree monomial ideal
$$ \idealB_{c^{(0)}} \quad = \quad \bigl\langle \,
\prod_{i \in V} c_i \,:\, V \,\hbox{encodes a vertex of} \, \,\mathcal{P}_{c^{(0)}} \,\bigr\rangle~.$$
The distinct combinatorial types of the polyhedra $\, \mathcal{P}_{c^{(0)}} \,$
determine a  natural {\em chamber decomposition} of
the cone $\mathcal{Q}$ into finitely
many smaller cones: if two polyhedra $\mathcal{P}_{c^{(0)}}$ and $\mathcal{P}_{d^{(0)}}$ correspond to points in such a chamber of the decomposition, then the polyhedra have the same set of supports $V$ of their vertices.  For an example see Figure~\ref{fig:Q}.
In the context of chemical reaction networks, the same chamber decomposition
appeared in  recent work  of Craciun {\em et al.} \cite{CPR}.
Specifically, its chambers were denoted $S_i$ in \cite[\S 2.1]{CPR}.

The ideal $\idealB_{c^{(0)}}$ depends only
on the chamber that contains the image of $c^{(0)}$.
For any subset $\Omega \subset \R^s_{> 0}$, we take
the sum of the ideals corresponding to all chambers that
intersect the image of $\Omega$ in $\mathcal{Q}$. That sum is the ideal
$$ \idealB_{\Omega}  \quad = \quad \bigl\langle \,
\prod_{i \in V} c_i \,:\, V \,\hbox{encodes a vertex of} \,\, \mathcal{P}_{c^{(0)}} \,\,
\hbox{for some} \,\, c^{(0)} \in \Omega\, \bigr\rangle~. $$
The above ideals are considered either in the polynomial
ring $  \mathbb{Q}[c_1,\ldots,c_s]$ or in its quotient
$R = \mathbb{Q}[c_1,\ldots,c_s]/\langle c_1 c_2 \cdots c_s \rangle$,
depending on the context.

Let $\idealI_1$ and $\idealI_2$ be two arbitrary ideals in $R$. Recall (e.g.~from \cite{CLO}) that
the {\em saturation} of $\idealI_1$ with respect to $\idealI_2$ is a new ideal that contains $\idealI_1$, namely,
$$ {\rm Sat}(\idealI_1,\idealI_2) \,\, = \,\,
(\idealI_1 : \idealI_2^\infty) \,\, = \,\, \bigl\{ f \in R \,:\,
f \cdot (\idealI_2)^m \subseteq \idealI_1 \,\hbox{for some} \,m \in \mathbb{Z}_{>0} \bigr\}. $$
Here, we shall be interested in the following nine saturation ideals:
\begin{equation}
\label{nineideals} \begin{matrix}
{\rm Sat}(\idealI_G, \idealB) , & \,\, {\rm Sat}(\idealI_G, \idealB_{c^{(0)}}), & \,\, {\rm Sat}(\idealI_G, \idealB_\Omega), \\\
{\rm Sat}(\idealJ_G, \idealB) , & \,\, {\rm Sat}(\idealJ_G, \idealB_{c^{(0)}}), & \,\, {\rm Sat}(\idealJ_G, \idealB_\Omega), \\\
{\rm Sat}(\idealM_G, \idealB) , & \,\, {\rm Sat}(\idealM_G, \idealB_{c^{(0)}}), & \,\, {\rm Sat}(\idealM_G, \idealB_\Omega).
\end{matrix}
\end{equation}
The following theorem is a refinement of our result in Theorem \ref{thm:Main}.

\begin{thm} \label{thm:SecondMain}
The relevant minimal siphons of $G$ are the inclusion-minimal sets 
$\{i \in [s]: c_i \in \idealP \}$ where $\idealP$ runs over minimal primes from (\ref{nineideals}).
The ideals in the first, second, and third columns yield relevant siphons,
$c^{(0)}$-relevant siphons, and  $\Omega$-relevant siphons, respectively.
The ideals in the first row are for all networks $G$, those in the third row for strongly
connected $G$, and those in the middle  row for $G$ with
 strongly connected components.
\end{thm}

\begin{proof}
The variety of the ideal ${\rm Sat}(\idealI_1,\idealI_2)$ is the union of
all irreducible components of the variety $V(\idealI_1)$ that do not lie in $V(\idealI_2)$.
The result now follows from Theorem \ref{thm:Main} and the following
observations. The non-negative variety $V_{\geq 0}(\idealB)$ consists of all points in $\R^s_{\geq 0}$
whose image modulo $L_{\rm stoi}$ lies in the boundary of the cone $\mathcal{Q}$.
Thus, for a minimal siphon $Z$, the image of the variety $V_{\geq 0}(\idealP_Z)$ is in the boundary of 
$\mathcal{Q}$ if and only if $Z$ is not relevant.  More precisely, the image of $V_{\geq 0}(\idealP_Z)$ is in the interior of the subcone spanned by $\{ a_i ~:~ i \notin Z\}$, so there exists a facet of $\mathcal{Q}$ that contains the subcone if and only if $Z$ is not relevant.  Therefore, any
irreducible component of $V(\idealJ_G)$ (or $V(\idealI_G)$ or $V(\idealM_G)$)  defines a non-relevant siphon $Z$ if and only if it lies in $V (\idealP_{F^{ c}})$ for some facet $F$ of $\mathcal{Q}$, which is equivalent to lying in $V(\idealB)$.
 
Next, the variety $V_{\geq 0}(\idealB_{c^{(0)}})$ is the union of all faces of 
the orthant $\R^s_{\geq 0}$ that are disjoint from the invariant polyhedron
$\mathcal{P}_{c^{(0)}}$.  So, for a minimal siphon $Z$, the ideal $\idealP_Z$ does not contain $\idealB_{c^{(0)}}$ if and only if there exists a vertex of $\mathcal{P}_{c^{(0)}}$ whose zero-coordinate set contains $Z$, which is equivalent to the condition that the face $F_Z$ of the polyhedron is non-empty.  Hence, any component of $V(\idealJ_G)$ (or $V(\idealI_G)$ or $V(\idealM_G)$) that defines a minimal siphon $Z$ lies in $V\left(\idealB_{c^{(0)}}\right)$ if and only if $Z$ is not relevant.  
Finally, the variety $V_{\geq 0}(\idealB_\Omega)$ is the
intersection of the varieties $V_{\geq 0}(\idealB_{c^{(0)}})$ as $c^{(0)}$ runs over $\Omega$.
\end{proof}

\begin{example}
In Example \ref{ex:2} and Example \ref{ex:3}, 
$\mathcal{Q}$ is the cone over a triangle, and the three minimal 
siphons are precisely the facets of that triangular cone. Thus, there are no
relevant siphons at all. This is seen algebraically by
verifying the identities  $\,{\rm Sat}(\idealJ_G,\idealB) = \langle 1 \rangle $  and
${\rm Sat}(\idealI_G,\idealB) = \langle 1 \rangle$. \qed
\end{example}

We now discuss the case when a network has no relevant siphons, by making
the connection to  work of Angeli {\em et al.} \cite{PetriNetExtended}, which focuses on chemical reaction networks whose siphons $Z$ all satisfy the following condition: 
\begin{align*} (\bigstar) & \text{ there exists a non-negative conservation relation $l \in L_{\rm cons} \cap \R^s_{\geq 0}$}  \\ & \text{ whose support ${\rm supp}(l) = \{ i \in [s]: l_i > 0 \}$ is a subset of $Z$.} 
\end{align*}

Recall that Angeli~{\em et al.}\ call siphons satisfying this property ``structurally non-emptiable''  \cite[\S 8]{PetriNetExtended}.  Note that the property ($\bigstar$) needs only to be checked for minimal siphons in order for all siphons to satisfy the property.  For some chemical reaction systems, such as toric dynamical systems (including Examples~\ref{ex:1} and~\ref{ex:2}), this property is sufficient for proving persistence \cite{Anderson08,AndShiu09,PetriNetExtended,SiegelMacLean},
and what was offered in this section are elegant and efficient algebraic tools for
deriving such proofs. 

\begin{lem}\label{prop:non-relevant=star}
For a chemical reaction network $G$, a siphon $Z$
satisfies property ($\bigstar$) if and only if $Z$
 is not relevant (which is equivalent to~$\idealB \subseteq \idealP_Z$).
 
\end{lem}
\begin{proof}
The ``only if'' direction is clear.  For the ``if'' direction, let $Z$ be a non-relevant siphon.  As usual, for $\sigma:= \dim L_{\rm stoic}$, we fix a matrix $\A \in \R^{(s-\sigma) \times s}$ whose rows span $L_{\rm cons}$, and we identify $\mathcal{Q}$ with the cone spanned by the columns $a_{i}$ of $\A$.  Let $F$ be a facet of $\mathcal{Q}$ that contains the image of $V_{\geq 0}(\idealP_{Z})$, and let $v\in \R^{s-\sigma}$ be a vector such that the linear functional $\langle v, - \rangle$ is zero on $F$ and is positive on points of $\mathcal{Q}$ outside of $F$.  The vector $l:= v\A$ is in $ L_{\rm cons}$, and we claim that this is a non-negative vector as  in ($\bigstar$).  Indeed, $l_{i}= \langle v, a_{i} \rangle$ is zero if $i \in F$ and is positive if $i \notin F$, and thus, ${\rm supp}(l) =F^{c} \subseteq Z$.  
\end{proof}

The following result extends Theorem~2 in Angeli {\em et al.} \cite{PetriNetExtended}.

\begin{thm} \label{thm:all_not_relevant}
None of the siphons of the network $G$ are relevant
 if and only if
${\rm Sat}(\idealI_G,\idealB) = \langle 1 \rangle$
if and only if all siphons satisfy property ($\bigstar$).  In this case, none of the
  invariant polyhedra $\mathcal{P}_{c^{(0)}}$ has a boundary steady state.
 \end{thm}
\begin{proof}
The first claim follows from Lemma~\ref{prop:non-relevant=star} above.  
The second claim follows from the definition of relevant siphons and Lemma~\ref{lem:bdySiphon}.
\end{proof}

We next present a characterization of the ideals $\idealB$
and $\idealB_{c^{(0)}}$ in terms of combinatorial commutative algebra.
This allows us to compute these ideals entirely within 
a computer algebra system (such as {\tt Macaulay 2} \cite{M2}),
without having to make any calls to 
polyhedral software (such as {\tt polymake}). We assume a subroutine that computes
the largest monomial ideal contained in a given binomial ideal in the
 polynomial ring $\R[c_1,\ldots,c_s]$.
Let $\idealI_{\rm stoi} $ and $\idealI_{\rm cons}$ denote the lattice ideals
associated with the subspaces $L_{\rm stoi}$ and $L_{\rm cons}$.
These ideals are generated by the binomials
$c^{u_+} - c^{u_-}$ where $u = u_+  - u_-$ runs
over all vectors in $\Z^s$ that lie in the respective subspace.  
Here, $u_{+} \in \Z^s_{\geq 0}$ and $u_{-}\in \Z^s_{\geq 0}$ denote the positive and negative parts of a vector $u$ in $\Z^s$.

\begin{alg} \label{algorithm}
The ideals  $\idealB$ and $\idealB_{c^{(0)}}$ 
can be computed as follows:
\begin{enumerate}
\item The squarefree monomial ideal $\idealB$ is the radical of the largest monomial ideal
contained in  $\idealI_{\rm stoi} + \langle c_1 c_2 \cdots c_s \rangle$.
\item  The squarefree monomial ideal  $\idealB_{c^{(0)}}$ is  Alexander dual to the radical of the largest monomial ideal
contained in the initial ideal ${\rm in}_{c^{(0)}}(\idealI_{\rm cons})$.
\item If $c^{(0)}$ is generic (i.e.~the polyhedron $\mathcal{P}_{c^{(0)}}$ is simple)
then  the radical of ${\rm in}_{c^{(0)}}(\idealI_{\rm cons})$
is a monomial ideal, and its Alexander dual equals $\idealB_{c^{(0)}}$.
\end{enumerate}
\end{alg}

The correctness of part 1 rests on the fact that the zero set of
the lattice ideal $\idealI_{\rm stoi}$ is precisely the affine toric variety
associated with the cone $\mathcal{Q}$. Adding the principal
ideal $\langle c_1 c_2 \cdots c_s \rangle$ to $\idealI_{\rm stoi}$ is equivalent to
taking the image of $\idealI_{\rm stoi}$ in $R$. The nonnegative
variety of the resulting ideal is the union of all
faces of $\R^s_{\geq 0}$ whose image modulo 
$L_{\rm stoi}$ is in the boundary of~$\mathcal{Q}$.

For parts 2 and 3 we are using concepts and results from the
textbook \cite{MS}. The key idea is to use the
initial concentration vector $c^{(0)}$  as a partial term order.
 Initial ideals of lattice ideals are discussed in 
\cite[\S 7.4]{MS}. Alexander duality of squarefree monomial ideals
is introduced in \cite[\S 5.1]{MS}. The correctness of part 3 is an immediate corollary to
\cite[Theorem 7.33]{MS}, and part 2 is derived from part 3 by  a perturbation argument.  In the next section, we demonstrate how to compute all these ideals in {\tt Macaulay 2}.
\section{Computing siphons in practice} \label{sec:ComputingSiphons}

We start with a network that has both relevant and non-relevant siphons.
This example serves to illustrate the various results in the previous section.

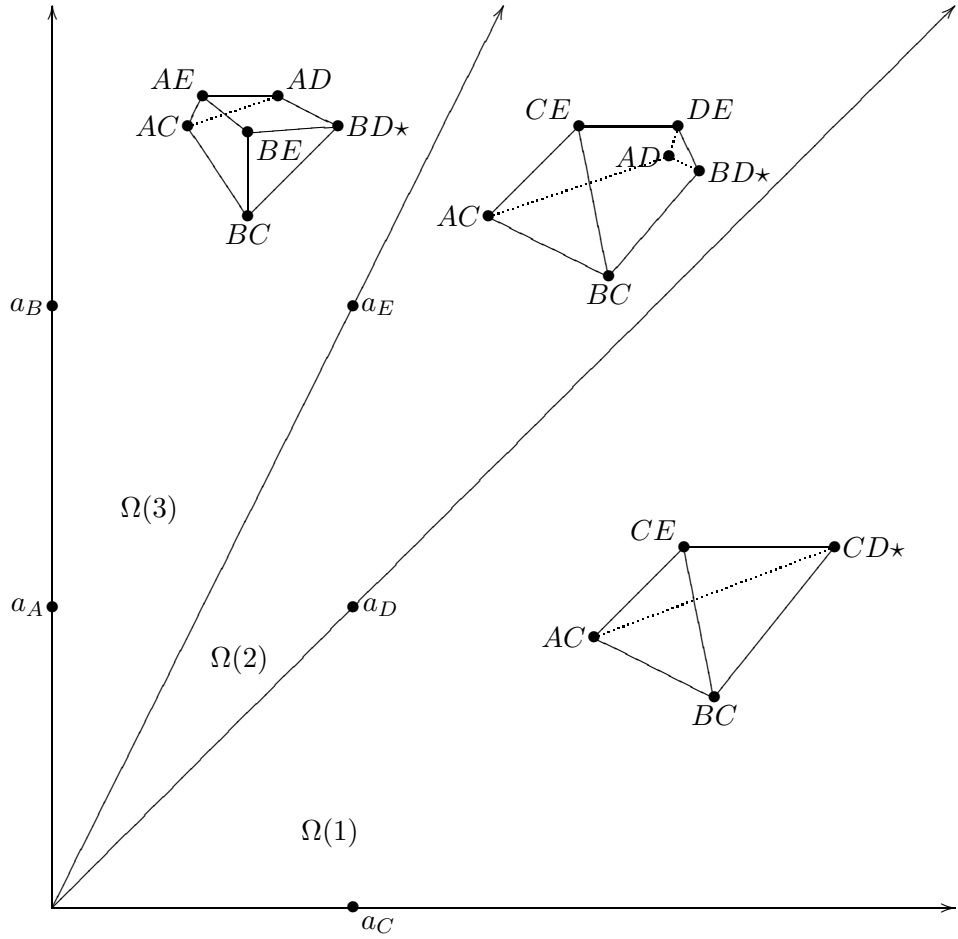
\begin{figure}[htbp]
\begin{center}
\[
 \begin{xy}<40mm,0cm>:
 (0,0) ="origin" *!UR{}*{}; 
 (1,0) ="C" *+!UL{a_{C}} *{\bullet}; 
 (3,0) ="3C" *!UR{}*{};
 (1,1) ="D" *!L{~a_{D}}*{\bullet};
 (3,3) ="moreD";
 (1.5,3)="moreE"; 
 (1,2) ="E" *+!L{a_{E}} *{\bullet};
 (0,1) ="A" *+!R{a_{A}} *{\bullet}; 
 (0,2) ="B" *+!R{a_{B}} *{\bullet}; 
 (0,3) ="3A" *+!R{ } *{}; 
 (.8,.25) ="region1" *+!L{\Omega(1)} *{}; 
 (.5,.75) ="region2" *+!LD{\Omega(2)} *{}; 
 (.2,1.25) ="region3" *+!LD{\Omega(3)} *{}; 
 (1.45,2.1) ="anchor2"; 
 "anchor2"+(0,.2) = "AC(2)" *+!R{AC} *{\bullet};
 "anchor2"+(.3,.5) = "CE(2)" *+!DR{CE} *{\bullet};
 "anchor2"+(.4,0) = "BC(2)" *+!U{BC} *{\bullet};
 "anchor2"+(.6,.4) = "AD(2)" *+!R{AD} *{\bullet};
 "anchor2"+(.63,.5) = "DE(2)" *+!DL{DE} *{\bullet};
 "anchor2"+(.7,.35) = "BD(2)" *+!L{BD \star} *{\bullet};
 "AC(2)";"CE(2)" **@{-};
 "BC(2)";"CE(2)" **@{-};
 "AC(2)";"BC(2)" **@{-};
 "DE(2)";"CE(2)" **@{-};
 "BC(2)";"BD(2)" **@{-};
 "AC(2)";"AD(2)" **@{.};
 "BD(2)";"AD(2)" **@{.}; 
 "AD(2)";"DE(2)" **@{.};
 "DE(2)";"BD(2)" **@{-};
 (0.25,2.3) ="anchor3"; 
 "anchor3"+(.2,.3) = "AC(3)" *+!R{AC} *{\bullet};
 "anchor3"+(.4,0) = "BC(3)" *+!U{BC} *{\bullet};
 "anchor3"+(.5,.4) = "AD(3)" *+!DL{AD} *{\bullet};
 "anchor3"+(.7,.3) = "BD(3)" *+!L{BD \star} *{\bullet};
 "anchor3"+(.25,.4) = "AE(3)" *+!DR{AE} *{\bullet};
 "anchor3"+(.4,.28) = "BE(3)" *+!UL{BE} *{\bullet};
 "AC(3)";"BC(3)" **@{-};
 "BC(3)";"BD(3)" **@{-};
 "AC(3)";"AD(3)" **@{.};
 "BD(3)";"AD(3)" **@{-}; 
 "AC(3)";"AE(3)" **@{-};
 "AE(3)";"AD(3)" **@{-};
 "AE(3)";"BE(3)" **@{-};
 "BE(3)";"BC(3)" **@{-};
 "BE(3)";"BD(3)" **@{-};
 (1.8,0.7) ="anchor1"; 
  "anchor1"+(0,.2) = "AC(1)" *+!R{AC} *{\bullet};
 "anchor1"+(.3,.5) = "CE(1)" *+!DR{CE} *{\bullet};
 "anchor1"+(.4,0) = "BC(1)" *+!U{BC} *{\bullet};
 "anchor1"+(.8,.5) = "CD(1)" *+!L{CD \star} *{\bullet};
 "AC(1)";"CE(1)" **@{-};
 "BC(1)";"CE(1)" **@{-};
 "AC(1)";"BC(1)" **@{-};
 "CD(1)";"CE(1)" **@{-};
 "BC(1)";"CD(1)" **@{-};
 "AC(1)";"CD(1)" **@{.};
   {\ar "origin"*{};"3C"+(0,0)*{}};
   {\ar "origin"*{};"3A"+(0,0)*{}};
   {\ar "origin"*{};"moreD"+(0,0)*{}};
   {\ar "origin"*{};"moreE"+(0,0)*{}};
    \end{xy}
 \]
 \caption{The chamber decomposition of the cone $\mathcal{Q}$ for the network
 in Example~\ref{ex:1}.  The cone is spanned by the columns of the matrix $\A$ in (\ref{matrixA}). 
 Each of the three maximal chambers $\Omega(1)$, $\Omega(2)$, and $\Omega(3)$
 contains a picture of the corresponding $3$-dimensional polyhedron $\mathcal{P}_{c^{(0)}}$. The vertices of each polyhedron are labeled by their supports. The star ``$\star$'' indicates the unique vertex steady state, which arises from the siphon $\{A,B,E\}$ or $\{A,C,E\}$.  \label{fig:Q} }
\end{center}
\end{figure}

\begin{example} \label{ex:Return}
We return to the chemical reaction network in Example~\ref{ex:1}.  The sums ${C}+{D}+ {E}$ 
 and ${A} + 2{B} + D + 2E $ are both constant along trajectories.  
Chemically, this says that both the total amount of free and bound forms of the ligand and the total amount of the free and bound forms of the receptor remain constant.  Thus, the matrix $\A$ can be taken to be
 \begin{align} \label{matrixA}
 \A \quad = \quad  \bigl(a_A,a_B,a_C,a_D,a_E\bigr) \quad  = \quad
 \left( 
 \begin{array}{ccccc} 
 0 & 0 & 1 & 1& 1 \\
 1 & 2 & 0 & 1 & 2 \\
 \end{array}
 \right) ~.
 \end{align}
 The two rows of $\A$ form a basis of $L_{\rm cons}$.
   The cone $\mathcal{Q}$ is spanned by the columns of $\A$.  The chamber decomposition of $\mathcal{Q}$ is depicted in Figure~\ref{fig:Q}.  We see that the two facets of $\mathcal{Q}$ define the following ideal of $\mathbb{Q}[A,B,C,D,E]$:
\[
\idealB \quad = \quad  \langle C,~D,~E \rangle \cap \langle A,~B,~D,~E \rangle~
\quad = \quad \langle \,AC, ~ BC,  ~D, ~E \,\rangle~.
\]
The relevant siphons are derived from
$\idealM_G = \langle A^2C, AD, E, BC \rangle$ as follows:
$$ {\rm Sat}(\idealM_G, \idealB) \,\,\,= \,\,\,\langle A, BC, E \rangle
\,\,= \,\, \langle A,B,E \rangle  \,\cap \,\langle A,C,E \rangle~. $$
Thus two of three minimal siphons in Example \ref{ex:1} are relevant.
The third siphon is not relevant as its ideal $\langle C, D,E \rangle$
contains $\idealB$. This corresponds to the fact, seen in Figure~\ref{fig:Q}, that the vectors
$a_A$ and $ a_B$ span a facet of $\mathcal{Q}$.

The chamber decomposition of $\mathcal{Q}$ consists of
 three open chambers $\Omega (1)$, $\Omega (2)$, and $\Omega (3)$, and two rays $\Omega (12)$ and $\Omega (23)$ between the three chambers.  These five chambers are encoded in the following ideals, whose generators can be read off from the vertex labels of the polyhedra 
 $\mathcal{P}_\Omega$ in Figure~\ref{fig:Q}:
 \begin{align*}
\idealB_{\Omega (1) } \quad &= \quad \langle C D,~ C E,~ A C,~ B C \rangle~, \\
\idealB_{\Omega (12) } \quad &= \quad \langle D,~ C E,~ A C,~ B C \rangle~, \\
\idealB_{\Omega (2) } \quad &= \quad \langle A D,~ B D,~ D E,~ C E,~ A C,~ B C \rangle~, \\
\idealB_{\Omega (23) } \quad &= \quad \langle  A D,~ B D,~ E,~ A C,~ B C \rangle~, \\
\idealB_{\Omega (3) } \quad &= \quad \langle A D,~ B D,~ A E,~ B E,~ A C,~ B C \rangle~.
\end{align*}
For each chamber $\Omega$, the  ideal  $\,{\rm Sat}(\idealM_G, \idealB_{\Omega})\,$
reveals the $\Omega$-relevant siphons.
 We find that $\langle A,B,E \rangle$ is $\Omega (1)$-  and $\Omega (12)$-relevant, 
and that $\langle A,C,E \rangle$ is $\Omega (12)$-,  $\Omega (2)$-, $\Omega (23)$-, and $\Omega (3)$-relevant.  These two siphons define a unique vertex steady state on each invariant polyhedron $\mathcal{P}_{c^{(0)}}$. Note that the vertices $F_{\{A,B,E\}}$ and $F_{\{A,C,E\}}$ coincide for polyhedra along the ray $\Omega(12)$. \qed
\end{example}

The need for efficient algorithms for computing minimal siphons has been emphasized by Angeli {\em et al.} \cite{PetriNetExtended}, who argued that such an algorithm would allow quick verification of the hypotheses of Theorem~\ref{thm:all_not_relevant}.  Cordone {\em et al.} introduced one algorithm for computing minimal siphons in \cite{EnumSiphons}. We advocate Theorem~\ref{thm:Main}
as a new method for computing all minimal siphons, and Algorithm~\ref{algorithm} 
as a direct method for identifying relevant siphons.
Rather than implementing any such algorithm from scratch, it is convenient to
harness existing tools for monomial and binomial primary decomposition 
 \cite{DMM, ES}. We recommend
the widely-used computer algebra system {\tt Macaulay 2} \cite{M2}, 
and the implementations developed by  Kahle \cite{KahleM2} and Roune \cite{Rou}.

\smallskip

In what follows we show some snippets of {\tt Macaulay 2} code,
and we discuss how they are used to compute (relevant) minimal siphons of
small networks. Thereafter we examine two larger examples, which illustrate
the efficiency and speed of  monomial and binomial primary decomposition.
These examples support our view that 
 the algebraic methods of Section 3 are competitive for networks 
 whose size is relevant for research in systems biology.

\begin{example}
The following {\tt Macaulay 2} input uses the command {\tt decompose} to output the minimal primes for the three examples in the Introduction.  \begin{verbatim}
-- Example 1.1
R1 = QQ[A,B,C,D,E]; 
M = ideal(A^2*C, A*D, E, B*C);
decompose(M) 

-- Example 1.2
R2 = QQ[e,i,p,q,r,s];
I = ideal(s*e-q, q-p*e, q*i-r);
decompose (I + ideal product gens R2) 

-- Example 1.3
R3 = QQ[E,F,P,S_0,X,Y];
J = ideal(E*S_0-X, X*(E*P-X), F*P-Y, Y*(F*S_0-Y));
decompose (J + ideal product gens  R3)
\end{verbatim} 
By Theorem~\ref{thm:Main}, the minimal siphons can be read off from the primes. 
\qed
\end{example}

\begin{example}
We return to the chemical reaction network of Examples~\ref{ex:1} and~\ref{ex:Return}.  The following {\tt Macaulay 2} code utilizes item 2 in Algorithm~\ref{algorithm}.
 \begin{verbatim}
-- Example 1.1: c0-relevant siphons 
c0 = {0,0,1,1,0};
R = QQ[A,B,C,D,E, Weights => c0];
IG  = ideal(A^2*C-A*D, A*D-E, E-B*C, A*B*C*D*E);
ICons = ideal(C*D*E-1, A*B^2*D*E^2-1);  
Bc0 = dual radical monomialIdeal leadTerm ICons;
decompose saturate(IG,Bc0)
\end{verbatim} 
In the first line, the vector $c^{(0)}$ was chosen to represent a point in the
chamber $\Omega(1)$, so the output is the unique
$\Omega(1)$-relevant minimal siphon. \qed
\end{example}

The next example is of a large strongly-connected chemical reaction, and the computation shows the 
power of monomial primary decomposition.  

\begin{example}
Consider the  following strongly connected network which is comprised of $s$ species,
$s-1$ complexes, and $s-2$ reversible reactions:
\begin{align*}
c_{1} c_{2} \, \leftrightarrows \, c_{2} c_{3} \, \leftrightarrows 
\, c_{3} c_{4} \, \leftrightarrows 
\, \cdots \, \leftrightarrows \,c_{s-1} c_{s} ~.
\end{align*}
The number of minimal siphons satisfies the recursion $N(s)=N(s-2)+N(s-3)$, where $N(2)=2$, $N(3)=2$, and $N(4)=3$.  For $s=50$ species we obtain $N(50)= 1,221,537 $.
The following  {\tt Macaulay 2} code verifies this:
\begin{verbatim}
s = 50
R = QQ[c_1..c_s];
M = monomialIdeal apply(1..s-1,i->c_i*c_(i+1));
time betti gens dual M
\end{verbatim}
We now explain the commands that are used above.  First, 
${\tt M}$ is the monomial ideal $\idealM_G$ generated by complexes, and {\tt dual} outputs 
its Alexander dual \cite{MS}, which is the monomial ideal whose generators are the products of 
the species-variables in any minimal siphon.  Secondly, {\tt betti}  applied to {\tt gens dual M} outputs the degrees of all the generators of {\tt dual M}; these degrees are exactly the sizes of all minimal siphons.  The command {\tt time} allows us to see that the computation of the minimal siphons takes  
only a few seconds.
  Displayed below is a portion the output of the last command above; the list tells the number of minimal siphons of each possible size.  
\begin{verbatim}
            0       1
o5 = total: 1 1221537
         0: 1       .
         1: .       .
         2: .       .
	        ...	     
        23: .       .
        24: .      26
        25: .    2300
        26: .   42504
        27: .  245157
        28: .  497420
        29: .  352716
        30: .   77520
        31: .    3876
        32: .      18
\end{verbatim}
The current version of {\tt dual} in {\tt Macaulay 2} uses Roune's
implementation of his slice algorithm \cite{Rou}.
For background on the relation of Alexander duality and primary decomposition 
of monomial ideals, see the textbook \cite{MS}.
\qed
\end{example}

Our final example aims to illustrate the computation of minimal siphons for a larger
network with multiple strongly connected components.

\begin{example}
We here consider a chemical reaction network $G$ with $s=25$ species,
 $16$ bidirectional reactions and $32$ complexes.  
 The binomials representing the  $16$ reactions are the  adjacent $2 \times 2$-minors 
 of a $5\times 5$-matrix $(c_{ij})$,  and $\idealJ_G$ is the ideal generated by 
 these $16$  minors  $c_{i,j} c_{i+1,j+1} - c_{i,j+1} c_{i+1,j}$.
 For this network, $L_{\rm stoi}$ is the $16$-dimensional space consisting of
 all matrices whose row sums and column sums are zero, and
 $\mathcal{Q}$ is a $9$-dimensional  convex polyhedral cone, namely the cone
 over the product of simplices $\Delta_4 \times \Delta_4$.
 
   What follows is an extension of the results for adjacent minors
 of a $4 \times 4$-matrix  in \cite[\S 4]{DES}.
 The ideal $\idealJ_G$ is not radical. Using Kahle's software \cite{KahleM2},
 we found that it has $103$ minimal primes,
 of which precisely $26$ contribute minimal siphons
 that are relevant. Up to symmetry, these $26$ siphons fall into four symmetry classes, with representatives given by the following:
 \begin{align*}
Z_1 \quad &= \quad \{ c_{14}, c_{21}, c_{22}, c_{23}, c_{24}, c_{32}, c_{34}, c_{42}, c_{43}, c_{44}, c_{45}, c_{52} \}~, \\	
Z_2 \quad &= \quad \{ c_{14}, c_{21}, c_{22}, c_{23}, c_{24}, c_{33}, c_{34}, c_{35}, c_{41}, c_{42}, c_{43}, c_{53} \}~, \\ 
Z_3 \quad &= \quad \{ c_{14} , c_{24}, c_{31}, c_{32}, c_{33}, c_{34}, c_{42}, c_{43}, c_{44}, c_{45}, c_{52} \}~, \\ 
Z_4 \quad &= \quad \{ c_{14}, c_{24}, c_{31}, c_{32}, c_{33}, c_{34}, c_{43}, c_{44}, c_{45}, c_{53}  \}~. 
 \end{align*}
Under the group $D_8$ of reflections and rotations of the matrix $(c_{ij})$, the orbit of $Z_1$ consists of two siphons, and the orbits of $Z_2$, $Z_3$, and $Z_4$ each are comprised of eight siphons.  The corresponding four types of minimal primes have codimensions 13, 12, 12, and 12, and degrees 1, 2, 3, and 6.  

By randomly generating chambers, we found that, for every integer $r$ between $0$ and $26$,
other than $23$ and $25$,
there is a point $c^{(0)}$ in $\mathcal{Q}$ such that the number of $c^{(0)}$-relevant siphons 
is precisely $r$. We briefly discuss this for three initial conditions.  First, let $c^{(0)}$ be the all-ones matrix. Then  $\mathcal{P}_c^{(0)}$ is the {\em Birkhoff polytope}
which consists of all non-negative $5 \times 5$-matrices with row and column sums equal to five.  In this case, all $26$ minimal siphons are 
$c^{(0)}$-relevant: $Z_1$ defines a vertex, $Z_2$ and $Z_3$ define edges, and $Z_4$ defines a three-dimensional face of $\mathcal{P}_{c^{(0)}}$.
Next, consider the following initial conditions:
\begin{align*}
d^{(0)} ~ = ~ 
\left( 
 \begin{array}{ccccc} 
 1 & 1 & 1 & 1& 1 \\
 1 & 1 & 1 & 1& 1 \\
 1 & 1 & 1-\epsilon & 1& 1 \\
 1 & 1 & 1 & 1& 1 \\
 1 & 1 & 1 & 1& 1 \\
 \end{array}
 \right) \quad \text{and}  \quad\,
e^{(0)} ~ = ~ 
\left( 
 \begin{array}{ccccc} 
 1 & 1 & 1 & 1& 1 \\
 1 & 1 & 1 & 1& 1 \\
 1 & 1 & 1+\epsilon & 1& 1 \\
 1 & 1 & 1 & 1& 1 \\
 1 & 1 & 1 & 1& 1 \\
 \end{array}
 \right) \! ,
\end{align*}
where $\epsilon > 0$.  
Again, all $26$ minimal siphons are $d^{(0)}$-relevant, and
$F_{Z_1}$ is a vertex, $F_{Z_2}$ and $F_{Z_3}$ are edges of $\mathcal{P}_{d^{(0)}}$,
but now  $F_{Z_4}$ is a five-dimensional face.  
Finally, for initial condition $e^{(0)}$, only two minimal siphons
 are $e^{(0)}$-relevant, both in the class of $Z_1$,
 and they define vertices.  
 Thus, using the results of \cite{AndShiu09}, we can conclude
 that the system (\ref{CRN}) is persistent for $e^{(0)}$.
\qed
\end{example}

\section{Conclusion}
In the present paper, we gave a method that computes siphons and determines which of them are relevant.  To our knowledge, this is the first automatic procedure for checking the relevance of a siphon.  
As noted by Angeli {\em et al.} \cite{PetriNetExtended}, such a procedure is desirable for verifying whether large biochemical reaction systems are persistent.  Recall that persistence is the property that no species concentration tends to zero.  In practice, this corresponds to the observed behavior that a substrate that is present at the beginning of an experiment will also be present in some amount for all time.  
There is a class of systems for which the non-relevance of all siphons is a sufficient condition for such a system to be persistent, so our procedure can be used to verify quickly that a large network is persistent.  
  Such a class consists of toric dynamical systems \cite{TDS}.  Mathematically, the claim that toric dynamical systems are persistent is the content of the Global Attractor Conjecture, and we speculate that an algebraic approach to understanding siphons may be a step toward the conjecture.



\subsubsection*{Acknowledgments}  Anne Shiu was supported by a Lucent Technologies Bell Labs Graduate Research Fellowship.  Bernd Sturmfels was partially supported by the National
Science Foundation  (DMS-0456960 and DMS-0757236).  The authors acknowledge the helpful comments of anonymous reviewers, which greatly improved the paper.


\bigskip

\noindent {\bf Authors' addresses:}

\smallskip

\noindent Anne Shiu and Bernd Sturmfels, Dept.~of Mathematics,
University of California, Berkeley, CA 94720, USA,
{\tt \{annejls,bernd\}@math.berkeley.edu}


\begin{thebibliography}{}
	

\bibitem{Adleman}  L.~Adleman, M.~Gopalkrishnan, M.-D.~Huang, P.~Moisset, D.~Reishus, On the mathematics of the law of mass action,  {\tt arXiv:0810.1108}.


\bibitem{Anderson08} D.~Anderson (2008),  Global asymptotic stability for
  a class of nonlinear chemical equations,  SIAM J. Appl. Math.  68, 1464--1476.

\bibitem{AndShiu09}
D.~Anderson and A.~Shiu,  Persistence of deterministic population processes and the
  global attractor conjecture, SIAM J. Appl. Math., to appear, {\tt arXiv:0903.0901}.
 
\bibitem{AngeliSontag06}  
D.~Angeli and E.~Sontag (2006), Translation-invariant monotone systems, and a global convergence result for enzymatic futile cycles, Nonlinear Analysis Ser. B: Real World Applications 9, 128--140.
  
\bibitem{PetriNetExtended} 
D.~Angeli, P.~De Leenheer, and E.~Sontag (2007),
	A {P}etri net approach to persistence analysis in chemical reaction networks,  
	in I. Queinnec, S. Tarbouriech, G. Garcia, and S-I. Niculescu (eds), 
	Biology and Control Theory: Current Challenges (Lecture
  Notes in Control and Information Sciences Volume 357), Springer-Verlag, Berlin, 
  	pp.181--216.

\bibitem{ChavezThesis} 
M.~Chavez (2003), Observer design for a class of
  nonlinear systems, with applications to chemical and biological
  networks.  Ph.D. Thesis, Rutgers University.

\bibitem{EnumSiphons}
R.~Cordone, L.~Ferrarini, and L.~Piroddi (2005),
Enumeration algorithms for minimal siphons in Petri nets based on place constraints, IEEE Transactions on Systems, Man and Cybernetics, Part A 35(6), 844--854. 

\bibitem{CLO}
D.~Cox, J.~Little and D.~O'Shea (2007), Ideals, Varieties and Algorithms
3rd ed.,
Undergraduate Texts in Mathematics, Springer, New York.
    
\bibitem{TDS}
G.~Craciun, A.~Dickenstein, A.~Shiu, and B.~Sturmfels (2009),  Toric dynamical systems.
J.~Symbolic Computation 44(11), 1551--1565. 

\bibitem{CPR} G. Craciun, C. Pantea, and G. Rempala (2009),
Algebraic methods for inferring biochemical networks:\ a maximum likelihood approach.
Comput. Biol. Chem. 33(5), pp. 361--367.

\bibitem{DES} P. Diaconis, D. Eisenbud, and B. Sturmfels (1998), Lattice walks and primary decomposition, in B.E. Sagan and R.P. Stanley (eds), Mathematical Essays in Honor of Gian-Carlo Rota,  Birkh\"auser, Boston, pp. 173--193.

\bibitem{DMM} A.~Dickenstein, L.~Matusevich, and E.~Miller, 
Combinatorics of binomial primary decomposition. Math. Z., to appear, {\tt arXiv:0803.3846}.

\bibitem{ES} D.~Eisenbud and B.~Sturmfels (1996),
Binomial ideals, Duke Math.~J. 84, 1--45.

\bibitem{polymake}
E. Gawrilow and M. Joswig (2000), Polymake: a framework for                     
  analyzing convex polytopes, in Polytopes -- Combinatorics and Computation
  (Oberwolfach, 1997), DMV Sem., vol.~29, Birkh{\"a}user, Basel,
  pp.~43--73.

\bibitem{M2}
D. Grayson and M. Stillman: Macaulay 2, a software system for research in algebraic geometry,
{\tt www.math.uiuc.edu/Macaulay2/}.


\bibitem{KahleM2} T. 
 T. Kahle, Decompositions of binomial ideals, {\tt arXiv:0906.4873}. 

\bibitem{GeoMulti}
A. Manrai and J. Gunawardena (2008),
The geometry of multisite phosphorylation, Biophys. J. 95, 5533--5543.

  
  \bibitem{MS}
  E.~Miller and B.~Sturmfels (2005), Combinatorial Commutative Algebra,
  Graduate Texts in Mathematics, Springer, New York.

    
 \bibitem{Rou} B.~Roune (2009),
 The slice algorithm for irreducible decomposition of monomial ideals,
 J. Symbolic Comput. 44(4)  358--381.
 
\bibitem{SiegelMacLean}
D.~Siegel and D.~MacLean (2004),
Global stability of complex balanced mechanisms,  J. Math. Biol. 27, 89--110.


\bibitem{Sontag01} E.~Sontag (2001), Structure and stability of certain
  chemical networks and applications to the kinetic proofreading model
  of T-cell receptor signal transduction,  IEEE Trans.~Automat.~Control
 46, 1028--1047.

\end{thebibliography}
\end{document}